\documentclass[12pt]{article}
\usepackage{amssymb,amsthm,amsmath,latexsym}
\usepackage{url}
\newtheorem{thm}{Theorem}
\newtheorem{prop}[thm]{Proposition}
\newtheorem{lem}[thm]{Lemma}

\theoremstyle{remark}

\newcommand{\FF}{\mathbb{F}}

\newcommand{\0}{\mathbf{0}}
\newcommand{\1}{\mathbf{1}}
\newcommand{\ww}{\omega}
\newcommand{\vv}{\omega^2}
\newcommand{\cC}{\mathcal{C}}
\DeclareMathOperator{\wt}{wt}

\begin{document}
\title{Some optimal 
entanglement-assisted quantum codes
constructed from 
quaternary Hermitian linear complementary dual codes
}

\author{
Masaaki Harada\thanks{
Research Center for Pure and Applied Mathematics,
Graduate School of Information Sciences,
Tohoku University, Sendai 980--8579, Japan.
email: {\tt mharada@tohoku.ac.jp}.}
}
\date{}

\maketitle

\begin{abstract}
We establish the existence of
 optimal
 entanglement-assisted quantum 
$[[n,k,d;n-k]]_2$ codes 
for 
$(n,k,d)=(14,6,7)$,
$(15,7,7)$,
$(17,6,9)$,  
$(17,7,8)$, 
$(19,7,9)$ and 
$(20,7,10)$.
These codes are obtained from
quaternary Hermitian  linear complementary dual codes.
We also give some observation on the largest minimum weights.
\end{abstract}

\section{Introduction}

Let $\FF_q$ denote the finite field of order $q$,
where $q$ is a prime power.
In this note, an $[n,k]_q$ code means
a code over $\FF_q$ of length $n$ and dimension $k$.
The {\em Euclidean dual} code $C^{\perp}$ of an $[n,k]_q$ code $C$ 
is defined as
$
C^{\perp}=
\{x \in \FF_q^n \mid \langle x,y\rangle = 0 \text{ for all } y \in C\},
$
where $\langle x,y\rangle = \sum_{i=1}^{n} x_i {y_i}$
for $x=(x_1,\ldots,x_n), y=(y_1,\ldots,y_n) \in \FF_q^n$.
For any $x \in \FF_{q^2}$, the conjugation of $x$ is
defined as $\overline{x}=x^q$.
The {\em Hermitian dual} code $C^{\perp_H}$ of an $[n,k]_{q^2}$ code $C$ 
is defined as
$
C^{\perp_H}=
\{x \in \FF_{q^2}^n \mid \langle x,y\rangle_H = 0 \text{ for all } y \in C\},
$
where $\langle x,y\rangle_H= \sum_{i=1}^{n} x_i \overline{y_i}$
for $x=(x_1,\ldots,x_n), y=(y_1,\ldots,y_n) \in \FF_{q^2}^n$.
Let $\0_n$ denote the zero vector of length $n$.
A code $C$ over $\FF_q$ is called 
{\em Euclidean linear complementary dual}
if $C \cap C^\perp = \{\0_n\}$.
A code $C$ over $\FF_{q^2}$ is called 
{\em Hermitian linear complementary dual}
if $C \cap C^{\perp_H} = \{\0_n\}$.
These two families of codes are collectively called
{\em linear complementary dual} (LCD for short) codes.

LCD codes were introduced by Massey~\cite{Massey} and gave an optimum linear
coding solution for the two user binary adder channel.
Recently, much work has been done concerning LCD codes
for both theoretical and practical reasons.
In particular, Carlet, Mesnager, Tang, Qi and Pellikaan~\cite{CMTQ2}
showed that 
any code over $\FF_q$ is equivalent to some Euclidean LCD code
for $q \ge 4$ and
any code over $\FF_{q^2}$ is equivalent to some Hermitian LCD code
for $q \ge 3$.
This motivates us to study Euclidean LCD codes over $\FF_q$
$(q=2,3)$ and quaternary Hermitian LCD codes.
Here, we consider only the latter.
In addition, it is known that quaternary Hermitian LCD codes 
give 
entanglement-assisted quantum
$[[n,k,d;n-k]]_2$ codes 
(see e.g.~\cite{CZKarxiv}, \cite{LA}, 
\cite{LLG}, \cite{LLG17} and~\cite{LLGF}
for background material on entanglement-assisted quantum codes).
More precisely, 
if there is a Hermitian LCD $[n,k,d]_4$ code, 
then there is an
entanglement-assisted quantum 
$[[n,k,d;n-k]]_2$ code
(see e.g.~\cite{LLG}, \cite{LLG17} and~\cite{LLGF}).
From this point of view,
quaternary Hermitian LCD codes
play an important role in the study of
entanglement-assisted quantum $[[n,k,d;n-k]]_2$ codes.
Note that quaternary Hermitian LCD
codes are also called {\em zero radical codes} 
(see e.g.~\cite{LLG}, \cite{LLG17} and~\cite{LLGF}).

A Hermitian LCD $[n,k,d]_4$ code is called {\em optimal}
if there is no Hermitian LCD $[n,k,d']_4$ code 
for $d' > d$.
An entanglement-assisted quantum 
$[[n,k,d;c]]_2$ code is called {\em optimal}
if there is no entanglement-assisted quantum 
$[[n,k,d';c]]_2$ code for $d' > d$.
We denote the largest minimum weight $d$ by $d_Q(n,k)$. 
For $k \le n \le 20$,
the current state of knowledge about $d_Q(n,k)$ are listed
in~\cite[Table~II]{LA}  and~\cite[Table~6]{LLGF}.
Many optimal
entanglement-assisted quantum $[[n,k,d;n-k]]_2$ codes codes are constructed
from optimal quaternary Hermitian LCD codes.
As a contribution in this direction, in this note, 
we establish the existence of
optimal
entanglement-assisted quantum 
$[[n,k,d;n-k]]_2$ codes 
for 
\begin{equation*}\label{eq:p}
(n,k,d)=(14,6,7),
(15,7,7),
(17,6,9),  
(17,7,8), 
(19,7,9),
(20,7,10).
\end{equation*}
From~\cite[Table~II]{LA}  and~\cite[Table~6]{LLGF},
we determine 
the largest minimum weight as follows:
\begin{align*}
&
d_Q(14,6)= d_Q(15,7)=7, d_Q(17,6)= d_Q(19,7)=9, 
\\&
d_Q(17,7)= 8 \text{ and } d_Q(20,7)= 10.
\end{align*}
In addition, we establish the existence of 
an
entanglement-assisted quantum 
$[[20,8,9;12]]_2$ code.
We also give some observation on the largest minimum weights
for Hermitian LCD $[n,k]_4$ codes 
for $k=n-1,n-2$ and $n-3$.

All computer calculations in this note
were done by {\sc Magma}~\cite{Magma}.

\section{New optimal codes}\label{sec:2}

\subsection{Optimal quaternary Hermitian LCD codes}

We denote the finite field of order $4$
by $\FF_4=\{ 0,1,\ww , \vv  \}$, where $\omega^2 = \omega +1$.
A linear $[n,k]_4$ {\em code} $C$ 
is a $k$-dimensional vector subspace of $\FF_4^n$.
All codes  over $\FF_4$
in this note are linear.   
A code over $\FF_4$ is called {\em quaternary}. 
The {\em weight} $\wt(x)$ of a vector $x \in \FF_4^n$ is
the number of non-zero components of $x$.
A vector of $C$ is called a {\em codeword} of $C$.
The minimum non-zero weight of all codewords in $C$ is called
the {\em minimum weight} $d(C)$ of $C$. An $[n,k,d]_4$ code
is an $[n,k]_4$ code with minimum weight $d$.
Two $[n,k]_4$ codes $C$ and $C'$ are 
{\em equivalent} 
if there is an $n \times n$ monomial matrix $P$ over $\FF_4$ with 
$C' = \{ x P \mid x \in C\}$.  

Every $[n,k,d]_4$ code is equivalent to a code with generator
matrix of the form
$\left(
\begin{array}{cc}
I_k & A
\end{array}
\right)
$,
where $A$ is a $k \times (n-k)$ matrix
and $I_k$ denotes the identity matrix of order $k$.
Let $r_i$ be the $i$-th row of $A$.
Here, we may assume that $A$ satisfies the following 
conditions:
\begin{itemize}
\item[\rm (i)]
$r_1=(\0_{n-k-d+1}, \1_{d-1})$, where $\1_s$ denotes the
all-one vector of length $s$,
\item[\rm (ii)]
$\wt(r_i) \ge d-1$,
\item[\rm (iii)]
the first nonzero element of $r_i$ is $1$,
\item[\rm (iv)]
$r_1 < r_2 < \cdots < r_k$ if $d \ge 3$ and
$r_1 \le r_2 \le \cdots \le r_k$ if $d \le 2$,
\end{itemize}
where we consider some order $<$ on
the set of vectors of length $n-k$.
The set of matrices $A$ is constructed, row by row, under
the condition that the minimum weight of the 
$[n+m-k,m]_4$ code with generator matrix
\[
\left(
\begin{array}{ccccc}
      & r_1 \\
I_{m} & \vdots   \\
      & r_m \\
\end{array}
\right)
\]
is at least $d$ for each $m=2,3,\ldots,k-1$.
It is obvious that the set of all $[n,k,d]_4$
codes obtained in this approach
contains a set of all inequivalent $[n,k,d]_4$ codes.
It is known that
a quaternary code $C$ is Hermitian LCD if and only if
$G \overline{G}^T$ is nonsingular
for a generator matrix $G$ of $C$, where
$A^T$ and $\overline{A}$ denote the transposed
matrix and the conjugate matrix for a matrix $A$, respectively.
In addition, it is known that
a quaternary code $C$ is Hermitian LCD if and only if
$C^{\perp_H}$ is Hermitian LCD 
(see e.g.~\cite{CMTQ2} and~\cite{LLG}).

By the above approach, 
our exhaustive computer search found 
a Hermitian LCD code with
parameters 
$[15,7,7]_4$,
$[17,6,9]_4$,  
$[17,7,8]_4$ and 
$[20,7,10]_4$.
We denote these codes by
$C_{15}$,
$C_{17,1}$,
$C_{17,2}$  and
$C_{20}$, respectively.
These codes have generator matrices
$
\left(
\begin{array}{ccccc}
I_{7} &  M_{15}  \\
\end{array}
\right),
$
$
\left(
\begin{array}{ccccc}
I_{6} &  M_{17,1}  \\
\end{array}
\right)
$,
$
\left(
\begin{array}{ccccc}
I_{7} &  M_{17,2}  \\
\end{array}
\right)
$ and
$
\left(
\begin{array}{ccccc}
I_{7} &  M_{20}  \\
\end{array}
\right),
$
respectively, where
$M_{15}$,
$M_{17,1}$,
$M_{17,2}$  and
$M_{20}$ are listed in Figure~\ref{Fig}.

\begin{figure}[thbp]
\begin{center}
\begin{align*}
M_{15}=&
\left(
\begin{array}{ccccccccccccccc}
  0&  0&  1&  1&  1&  1&  1&  1\\
  1&  0&  1&\vv&\vv&  1&  0&\ww\\
  1&  1&  0&  1&\ww&  1&  1&  0\\
  1&  1&\ww&  1&\vv&  0&\ww&\vv\\
  1&\ww&  0&  1&  0&\ww&\vv&\vv\\
  1&\ww&\vv&\vv&  1&\ww&  1&  0\\
  1&\vv&\vv&  0&  0&\ww&\ww&\ww
\end{array}
\right)
\\
M_{17,1}=&
\left(
\begin{array}{ccccccccccccccc}
 0&  0&  0&  1&  1&  1&  1&  1&  1&  1&  1\\
 1&  1&  1&\ww&  1&  1&  1&  1&  0&  0&  0\\
 1&\ww&\ww&  1&\ww&  1&  0&  0&  1&  1&  0\\
 1&\ww&\vv&\ww&  1&  0&\ww&  0&  1&  0&  1\\
 1&\vv&  0&  1&  0&\ww&\vv&  1&  1&\ww&\ww\\
 1&\vv&\ww&\ww&  0&  0&\ww&  1&\ww&  1&  0
\end{array}
\right)
\\
M_{17,2}=&
\left(
\begin{array}{ccccccccccccccc}
  0&  0&  0&  1&  1&  1&  1&  1&  1&  1\\
  0&  1&\vv&\ww&  1&  0&\vv&\vv&\ww&  1\\
  0&  1&\ww&  0&  1&  1&  0&  1&\vv&\vv\\
  1&  0&\ww&\ww&  0&\ww&  1&\ww&  1&  0\\
  1&  0&\vv&  0&  1&  1&  0&\ww&\ww&\ww\\
  1&  1&  1&  0&\ww&\vv&\ww&  0&  1&\vv\\
  1&  1&\ww&\vv&\ww&\ww&  1&  0&\ww&  1
\end{array}
\right)
\\
M_{20}=&
\left(
\begin{array}{ccccccccccccccc}
 0&   0&   0&   0& 1& 1& 1&   1&  1&  1&  1&  1& 1\\
 1&   0& \vv& \ww& 1&\ww& 0& \ww&  1&\ww&  1&  0& 0\\
 1&   1&   0&   1& 0& 0&\ww& \vv&\ww&\ww&  1&\ww& 1\\
 1&   1& \ww& \ww&\ww& 1& 1&   0&  0&  1&  1&  0& 0\\
 1&   1& \vv&   1&\ww&\ww&\ww& \vv&  1&  0&\ww&  1& 0\\
 1& \ww&   1&   1& 1& 1& 1&   1&  1&  0&  0&  0& 0\\
 1& \vv& \ww& \vv& 1&\ww& 0&   1&  0&  1&  0&  1& 0
\end{array}
\right)
\end{align*}
\end{center}
\caption{Matrices $M_{15}$, $M_{17,1}$, $M_{17,2}$  and $M_{20}$}
\label{Fig}
\end{figure}

Let $C$ be an $[n,k]_4$ code.
A {\em shortened code} of $C$ 
on the coordinate $i$ is the set of all codewords 
in $C$ which are $0$ in the $i$-th coordinate with that 
coordinate deleted.
We denote the code by $S(C,i)$.
A {\em punctured code} of $C$ 
on the coordinate $i$ is the code obtained from $C$
by deleting the $i$-th coordinate.
Let $C_{14}$ be the code $S(C_{15},4)$.
We verified that
$C_{14}$ is a Hermitian LCD $[14,6,7]_4$ code.
We denote by $C_{19}$ the punctured code of $C_{20}$
on the first coordinate.
We verified  that
$C_{19}$ is a Hermitian LCD $[19,7,9]_4$ code.

Therefore, we have the following result.

\begin{prop}\label{prop}
There is a Hermitian LCD $[n,k,d]_4$ code for
\[
(n,k,d)=
(14,6,7),
(15,7,7),
(17,6,9), 
(17,7,8), 
(19,7,9),
(20,7,10).
\]
\end{prop}


The weight enumerator of an $[n,k]_4$ code
$C$ is defined as $\sum_{i=0}^n A_i y^i$,
where $A_i$ denotes the number of codewords of weight $i$ in $C$.
The weight enumerators of 
the codes $C_{14}$, $C_{15}$, 
$C_{17,1}$, 
$C_{17,2}$, 
$C_{19}$ and $C_{20}$ are listed in Table~\ref{Tab:WE}.

\begin{table}[thb]
\caption{Weight enumerators}
\label{Tab:WE}
\begin{center}
{\small
\begin{tabular}{c|l}
\noalign{\hrule height0.8pt}
Code & \multicolumn{1}{c}{Weight enumerator} \\
\hline
$C_{14}$ &
$1
+ 210y^{7}
+ 252y^{8}
+ 588y^{9}
+ 945y^{10}
+ 882y^{11}
+ 819y^{12}$\\&
$+ 336y^{13}
+  63y^{14}$
\\
$C_{15}$&
$
1
+  336y^{ 7}
+  756y^{ 8}
+ 1323y^{ 9}
+ 2415y^{10}
+ 4095y^{11}
+ 3759y^{12}
$\\
&$
+ 2289y^{13}
+ 1197y^{14}
+  213y^{15}$
\\
$C_{17,1}$&
$1
+ 201y^{ 9}
+ 279y^{10}
+ 492y^{11}
+ 777y^{12}
+ 840y^{13}
+ 849y^{14}
$\\
&$
+ 456y^{15}
+ 174y^{16}
+  27y^{17}$
\\
$C_{17,2}$&
$1
+  204y^{ 8}
+  549y^{ 9}
+ 1053y^{10}
+ 1977y^{11}
+ 3117y^{12}
+ 3711y^{13}
$\\
&$
+ 3111y^{14}
+ 1875y^{15}
+  642y^{16}
+  144y^{17}$
\\
$C_{19}$&
$1
+  111y^{ 9}
+  423y^{10}
+  801y^{11}
+ 1509y^{12}
+ 2595y^{13}
+ 3291y^{14}
$\\
&$
+ 3315y^{15}
+ 2502y^{16}
+ 1362y^{17}
+  402y^{18}
+   72y^{19}
$
\\
$C_{20}$&
$1
+  297y^{10}
+  441y^{11}
+  978y^{12}
+ 1767y^{13}
+ 2685y^{14}
+ 3381y^{15}
$\\
&$
+ 3078y^{16}
+ 2349y^{17}
+ 1038y^{18}
+  318y^{19}
+   51y^{20}$\\
\noalign{\hrule height0.8pt}
\end{tabular}
}
\end{center}
\end{table}


\subsection{Optimal
entanglement-assisted quantum codes}
An entanglement-assisted quantum
$[[n,k,d;c]]_2$ code $\cC$
encodes $k$ information qubits into $n$ channel qubits
with the help of $c$ pairs of maximally entangled Bell states.
The parameter $d$ is called the minimum weight of $\cC$.
The entanglement-assisted quantum code $\cC$
can correct up to $\lfloor \frac{d-1}{2} \rfloor$ 
errors acting on the $n$ channel qubits (see e.g.~\cite{LLG}
and~\cite{LLGF}).
An entanglement-assisted quantum 
$[[n,k,d;0]]_2$ code is a standard quantum code.
If there is a Hermitian LCD $[n,k,d]_4$ code, 
then there is an
entanglement-assisted quantum 
$[[n,k,d;n-k]]_2$ code 
(see e.g.~\cite{LLG} and~\cite{LLGF}).

An entanglement-assisted quantum 
$[[n,k,d;c]]_2$ code is called {\em optimal}
if there is no entanglement-assisted quantum 
$[[n,k,d';c]]_2$ code for $d' > d$.
We denote the largest minimum weight $d$ by $d_Q(n,k)$. 
The largest minimum weights $d_Q(n,k)$ have been widely studied
in~\cite{LA} for $n \le 20$.
The current state of knowledge about $d_Q(n,k)$ can be found 
in~\cite[Table~II]{LA}  and~\cite[Table~6]{LLGF} for $n \le 20$.
From~\cite[Table~II]{LA}  and~\cite[Table~6]{LLGF}, we have the following:
\[
\begin{array}{ll}
d_Q(14,6)=6\text{ or } 7, &d_Q(15,7)=6\text{ or } 7,\\
d_Q(17,6)=8\text{ or } 9, &d_Q(17,7)=7\text{ or } 8, \\
d_Q(19,7)=8\text{ or } 9, &d_Q(20,7)=9\text{ or }10.
\end{array}
\]
Therefore, 
from quaternary Hermitian LCD codes given in Proposition~\ref{prop},
we have the following: 

\begin{prop}
\begin{itemize}
\item[\rm (i)] 
	     There is an optimal
entanglement-assisted quantum 
$[[n,k,d;n-k]]_2$ code from a
Hermitian LCD $[n,k,d]_4$ code
for 
\[
(n,k,d)=(14,6,7),
(15,7,7),
(17,6,9),  
(17,7,8), 
(19,7,9),
(20,7,10).
\]
\item[\rm (ii)] 
${\displaystyle
\begin{array}{ll}
d_Q(14,6)= d_Q(15,7)=7, & d_Q(17,6)= d_Q(19,7)=9,\\
d_Q(17,7)= 8,           & d_Q(20,7)= 10.
\end{array}
}$
\end{itemize}
\end{prop}


Let $d_4(n,k)$ denote the largest minimum weight among all 
Hermitian LCD $[n,k]_4$ codes.
From~\cite[Table~II]{LA} and~\cite[Table~6]{LLGF}, it is known that
$d_4(14,6)\le 7$,
$d_4(15,7)\le 7$,
$d_4(17,6)\le 9$,
$d_4(17,7)\le 8$,
$d_4(19,7)\le 9$ and 
$d_4(20,7)\le 10$.
Hence,
quaternary Hermitian LCD codes listed in Proposition~\ref{prop}
are optimal.

\begin{prop}
${\displaystyle
\begin{array}{ll}
d_4(14,6)= d_4(15,7)= 7, &d_4(17,6)= d_4(19,7)= 9,\\
d_4(17,7)= 8,             &d_4(20,7)= 10.
\end{array}
}$
\end{prop}

\subsection{Largest minimum weights}

From~\cite[Table~II]{LA} and~\cite[Table~6]{LLGF}, 
it is known that $d_Q(12,6)=5$ or $6$.
By the approach given in the beginning of this section,
our exhaustive search shows that there is no Hermitian LCD 
$[12,6,6]_4$ code.
In addition, 
our exhaustive computer search found 
a Hermitian LCD $[12,6,5]_4$ code $D_{12}$.
The code $D_{12}$ has generator matrix
$
\left(
\begin{array}{ccccc}
I_{6} &  N_{12}  \\
\end{array}
\right),
$
where
\[
N_{12}=
\left(
\begin{array}{ccccccccccccccc}
  0&  0&  1&  1&  1&  1\\
  0&  1&  0&\vv&\ww&\vv\\
  1&  1&  0&\ww&  0&  1\\
  1&  1&\vv&\ww&\vv&\vv\\
  1&\ww&  0&\ww&\ww&\ww\\
  1&\vv&  0&  1&  1&  0
\end{array}
\right).
\]
The weight enumerator of $D_{12}$ is given by:
\[
1
+   72y^{5}
+  177y^{6}
+  378y^{7}
+  792y^{8}
+ 1044y^{9}
+  999y^{10}
+  522y^{11}
+  111y^{12}.
\]
\begin{prop}
$d_4(12,6)=5$.
\end{prop}

It is worthwhile to determine whether there is a 
entanglement-assisted quantum $[[12,6,6;6]]_2$ code.

From~\cite[Table~II]{LA} and~\cite[Table~6]{LLGF}, it is known 
that $d_Q(20,8)=8,9$ or $10$.
By the approach given in the beginning of this section,
our computer search found a Hermitian LCD 
$[20,8,9]_4$ code $D_{20}$.
The code $D_{20}$ has generator matrix
$
\left(
\begin{array}{ccccc}
I_{8} &  N_{20}  \\
\end{array}
\right),
$
where
\[
N_{20}=
\left(
\begin{array}{ccccccccccccccc}
0&   0&   0&   0& 1&   1& 1& 1& 1& 1& 1& 1\\
1&   1&   1& \ww& 1&   1& 0& 0& 1& 1& 0& 0\\
1&   1& \ww&   1&\ww&   0& 1& 0& 1& 0& 1& 0\\
1& \ww& \vv&   1& 1&   1& 1& 1& 0& 0& 0& 0\\
1& \ww&   1&   0&\ww&   0&\ww& 1& \ww& 1& 0& 0\\
1& \vv& \vv& \ww& 0& \ww& 0& 1& \ww& 0& 1& 0\\
1& \vv& \ww&   0& 1& \vv& 0& 1& 1& 0& 0& 1\\
1& \vv&   1& \vv& 1&   0&\ww&\ww& 0& \ww& 1& 1
\end{array}
\right).
\]
The weight enumerator of $D_{20}$ is given by:
\begin{multline*}
1
+   288y^{ 9}
+   714y^{10}
+  1725y^{11}
+  3888y^{12}
+  7272y^{13}
+ 11208y^{14}
\\
+ 13338y^{15}
+ 12423y^{16}
+  8640y^{17}
+  4446y^{18}
+  1377y^{19}
+   216y^{20}.
\end{multline*}
\begin{prop}
\begin{itemize}
\item[\rm (i)]
There is a Hermitian LCD $[20,8,9]_4$ code
and 
	     there is an
	     entanglement-assisted quantum 
$[[20,8,9;12]]_2$ code.
\item[\rm (ii)]
$d_4(20,8)=9 \text{ or }10$ and $d_Q(20,8)=9 \text{ or }10$.
\end{itemize}
\end{prop}


\section{$d_4(n,k)$ for  $k=n-1,n-2,n-3$}

In this section, we study $d_4(n,k)$ for $k=n-1,n-2,n-3$.

Let $C$ be an $[n,n-1]_4$ code.
We may assume without loss of generality that 
$C$ has generator matrix of the following form:
\[
\left(
\begin{array}{cccc}
 &       && a_1 \\
 &I_{n-1}&& \vdots\\
 &       && a_{n-1}
\end{array}
\right),
\]
where $a_i \in \{0,1\}$ $(i=1,2,\ldots,n-1)$
and $a=(a_1,a_2,\ldots,a_{n-1})$.
The matrix 
$
\left(
\begin{array}{cccccc}
\overline{a_1} & \cdots & \overline{a_{n-1}} & 1
\end{array}
\right)
$ is a generator matrix of $C^{\perp_H}$.
It follows that $C^{\perp_H}$ is Hermitian LCD if and only if 
$\wt(a) \equiv 0 \pmod 2$.
In other words, 
$C$ is Hermitian LCD if and only if 
$\wt(a) \equiv 0 \pmod 2$.
Hence, we have the following:
\begin{prop}\label{prop:n-1}
Suppose that $n \ge 2$.  Then
\[
d_4(n,n-1)=
\begin{cases}
1 & \text{ if } n \text{ is even,}\\
2 & \text{ if } n \text{ is odd.}
\end{cases}
\]
\end{prop}

%

The following lemma is a key idea for the determination of 
$d_4(n,n-2)$ and $d_4(n,n-3)$.

\begin{lem}\label{lem:n-i}
Let $i$ be an integer with $2 \le i < n$.
If $n > \frac{4^i-1}{3}$, then
$d_4(n,n-i) = 2$.
\end{lem}
\begin{proof}
Let $C$ be an $[n,n-i,d]_4$ code with generator matrix
of the form:
\[
G=
\left(
\begin{array}{ccccccccc}
 &           && 1&1&0& \cdots & 0 \\
 &I_{n-i}&& \vdots& \vdots & \vdots & & \vdots\\
 &          && 1&1&0& \cdots & 0 \\
\end{array}
\right).
\]
Since $G \overline{G}^T=I_{n-i}$,
$C$ is Hermitian LCD.
By the construction, it is trivial that $C$ has minimum weight $2$.
By the sphere-packing bound, if $d \ge 3$, then
$n \le \frac{4^i-1}{3}$.
The result follows.
\end{proof}


\begin{prop}\label{prop:n-2}
\[
d_4(n,n-2)=
\begin{cases}
3& \text{ if } n =3, \\
2& \text{ if } n \ge 4.\\
\end{cases}
\]
\end{prop}
\begin{proof}
It is known that $d_4(3,1)=3$,
$d_4(4,2)=2$ and $d_4(5,3)=2$~\cite{LLGF}.
If $n \ge 6$, then $d_4(n,n-2)=2$ by Lemma~\ref{lem:n-i}.
\end{proof}

\begin{prop}\label{prop:n-3}
\[
d_4(n,n-3)=
\begin{cases}
3& \text{ if } 4 \le n \le 18, \\
2& \text{ if } n \ge 19.\\
\end{cases}
\]
\end{prop}
\begin{proof}
It is known that 
$d_4(n,n-3)=3$ for $n=4,5,\ldots,8$~\cite{LLGF}.
If $n \ge 22$, then $d_4(n,n-3)=2$ by Lemma~\ref{lem:n-i}.

It is known that the largest minimum weight among 
(unrestricted) $[n,n-3]_4$ codes is $3$ for $n=9,10,\ldots,21$.
By the approach given in the beginning of the previous section,
our exhaustive search shows that there is no 
Hermitian LCD $[n,n-3,3]_4$ code for $n=19,20,21$.
Let $E_{n}$ $(n=19,20,21)$
be the $[n,n-3]_4$ code with generator matrix
of the form:
\[
\left(
\begin{array}{ccccccccc}
 &           && 1&1&0 \\
 &I_{n-3}&&   \vdots& \vdots & \vdots\\
 &          && 1&1&0& \\
\end{array}
\right).
\]
As described in the proof of Lemma~\ref{lem:n-i},
$E_{n}$ $(n=19,20,21)$ is a Hermitian LCD 
$[n,n-3,2]_4$ code.

Let $E_{18}$ be the $[18,15]_4$ code with generator matrix
$
\left(
\begin{array}{ccccc}
I_{15} &  L_{18}  \\
\end{array}
\right),
$
where
\[
L_{18}^T=
\left(
\begin{array}{ccccccccccccccc}
 0&  0&  0&  1&  1&  1&  1&  1&  1&  1&  1&  1&  1&  1&  1\\
 1&  1&  1&  0&  0&  0&  1&  1&  1&\ww&\ww&\ww&\ww&\vv&\vv\\
 1&\ww&\vv&  1&\ww&\vv&  0&  1&\vv&  0&  1&\ww&\vv&\ww&\vv
\end{array}
\right).
\]
We define the codes $E_{i}$ $(i=17,16,\ldots,9)$ by
the shortened codes as follows:
\begin{multline*}
S(E_{18},1),
S(E_{17},2),
S(E_{16},1),
S(E_{15},4),
S(E_{14},1),
\\
S(E_{13},2),
S(E_{12},1),
S(E_{11},2),
S(E_{10},2),
\end{multline*}
respectively.
We verified that $E_n$
is a Hermitian LCD $[n,n-3,3]_4$ code
for $n=9,10,\ldots,18$.
The result follows.
\end{proof}

\bigskip
\noindent
{\bf Acknowledgment.}
This work was supported by JSPS KAKENHI Grant Numbers 15H03633
and 19H01802.
The author would like to thank the anonymous referee and
the editor Markus Grassl for the useful comments.



\end{document}